\documentclass[reqno,12pt,centertags]{amsart}

\usepackage{amsmath, amsthm, amscd, amssymb, latexsym, upref}

\usepackage{cite}
\input epsf
\usepackage{epsfig}
\usepackage{epic, eepicemu}
\usepackage{tikz}
\usepackage{mathrsfs}  

\newcommand{\bbB}{{\mathbb{B}}}
\newcommand{\bbC}{{\mathbb{C}}}
\newcommand{\bbN}{{\mathbb{N}}}
\newcommand{\bbR}{{\mathbb{R}}}
\newcommand{\bbZ}{{\mathbb{Z}}}

\newcommand{\cF}{{\mathcal{F}}}

\newcommand{\rmd}{{\mathrm{d}}}

\DeclareMathOperator{\imaginary}{Im}
\DeclareMathOperator{\real}{Re}
\renewcommand{\Re}{\real}
\renewcommand{\Im}{\imaginary}

\DeclareMathOperator{\tr}{tr}

\newcommand{\iop}{{\mathbf i}}

\allowdisplaybreaks 
\numberwithin{equation}{section}

\newtheorem{theorem}{Theorem}[section]
\newtheorem*{theorem-non}{Theorem}

\newtheorem{lemma}[theorem]{Lemma}
\newtheorem{proposition}[theorem]{Proposition}
\newtheorem{corollary}[theorem]{Corollary}

\theoremstyle{definition}

\newtheorem{remark}[theorem]{Remark}

\date{\today}
\title[Optimal Dispersion for Periodic Schr\"odinger Operators]{Optimal Dispersion for Discrete Periodic Schr\"odinger Operators}

\author[D. \ Damanik]{David Damanik}
\address{Department of Mathematics, Rice University, Houston, TX~77005, USA}
\email{ damanik@rice.edu}
\thanks{D.\ D.\ was supported in part by National Science Foundation grants DMS--2054752 and DMS--2349919}

\author[J.\ Fillman]{Jake Fillman}
\address{Department of Mathematics, Texas A\&M University, College Station, TX 77843,
USA}
\email{fillman@tamu.edu}
\thanks{J.\ F.\ was supported in part by National Science Foundation grant DMS--2513006.}

\author[G.\ Young]{Giorgio Young}
\address{Department of Mathematics,
  The University of Michigan,
 Ann Arbor, MI 48109, USA}
  \email[]{gfyoung@umich.edu}
  \thanks{G.\ Y.\  was supported by the National Science Foundation through grant DMS--2303363.}

\begin{document}

	\begin{abstract}
	We prove a dispersive estimate for periodic discrete Schr\"odinger operators on the line with optimal rate of decay.
 Additionally, by standard methods, we deduce dispersive estimates for the discrete nonlinear Schr\"odinger equation with small initial data and suitable nonlinearity when the underlying Hamiltonian is periodic.
	\end{abstract}
	
	\maketitle
	
\section{Introduction}

\subsection{Setting} 
We consider discrete Schr\"odinger operators $H = H_V:\ell^2(\bbZ) \to \ell^2(\bbZ)$ given by
\begin{equation} \label{eq:DSOdef}
[H\psi](n)
= \psi(n-1) + \psi(n+1) + V(n)\psi(n),
\end{equation}
where the \emph{potential} $V:\bbZ \to \bbR$ is bounded.

Operators of the form \eqref{eq:DSOdef} arise naturally in the context of a tight-binding approximation, and have been studied extensively over the years.
We point the reader to the textbooks \cite{CFKS, DF2022ESO1, DF2024ESO2, Teschl2000Jacobi} for background.
One is interested in the associated quantum dynamics, which are driven by the time-dependent Schr\"odinger equation:
\begin{equation}\label{eq:timeschrod}
{\iop} \frac{\partial \psi}{\partial t} = H_V \psi.
\end{equation}

There is a close relationship between notions of wave-packet spread and the spectral type of $H_V$. Most classically, this relationship is expressed through the RAGE theorem \cite{enss1978asymptotic,ruelle1969remark,amrein1973characterization}, linking qualitative notions of wave-packet spread with the spectral type of $H_V$ and providing the heuristic that wave-packet spread increases with the continuity of the spectral measure. Further evidence for this general heuristic was provided by Last \cite{Last1996JFA} in the context of ballistic transport, a more quantitative notion of wave-packet spread characterized by linear growth of the expectation of the position observable in time. In particular, his results imply that in 1D one has ballistic transport in a \emph{time-averaged sense} when the spectral type is absolutely continuous, while there are sub-ballistic lower bounds growing at the rate $T^{\alpha}$ on the time-averaged expectation of the position observable, where $0\leq \alpha\leq 1$ is such that the associated spectral is not singular with respect to $\alpha$-dimensional Hausdorff measure.

An important open question in the field centers on whether there is a connection of this type between ballistic transport without time-averaging, and absolute continuity; compare discussion in \cite{Simon1990CMP}. When $V$ is periodic, one has very strong spectral continuity, and a correspondingly strong form of ballistic transport is known to hold \cite{AschKnauf1998Nonlin, BoutetSabri2023OTAA, DLY2015CMP, Fillman2021OTAA}. On the other hand, in some well-understood one-dimensional models with absolutely continuous spectrum, ballistic transport has been shown for genuinely aperiodic models, including quasi-periodic operators in the KAM region \cite{GeKachkovskiy2023CPAM, Kachkovskiy2016CMP, Zhao2016CMP, Zhao2017JDE}, limit-periodic operators in the exponential regime \cite{Fillman2017CMP, Young2023JST} studied by Pastur and Tkachenko \cite{PasturTkachenko1, PasturTkachenko2}  and Egorova \cite{Egorova1987}, as well as quasi-periodic operators in the subcritical region \cite{ZhaoZhang2017CMP}, where the methods explicitly make use of the continuity properties of the spectral measures. For further background on ballistic transport, we refer the reader to \cite{DamMalYou2024JST}.

Given the unitarity of the evolution, spreading can also  be measured with dispersive estimates, such as estimates on the norm of the operator $e^{-{\iop}tH_V}$ as a map $\ell^1(\bbZ) \to \ell^\infty(\bbZ)$. In addition to capturing the dispersive feature often expected of wave-packets, the global control of the wavepacket such an estimate provides precludes a portion of the wavepacket translating linearly like a traveling wavefront, with the bulk remaining localized, a picture consistent with ballistic transport. 

Despite its implication on wave-packet spread, the link between spectral type and dispersion is less clear, with dispersive estimates (after projecting away from eigenstates) often being proven in regimes where decay conditions on the potential preclude singular continuous spectrum, or the estimates are proven after projecting onto the absolutely continuous subspace, compare \cite{schlag}. In the discrete setting, \cite{pelinovsky2008spectral} shows a dispersive estimate at the sharp rate after projecting onto the absolutely continuous spectrum for generic potentials of rapid enough decay. In the bounded but non-decaying regime, there is a dispersive estimate for periodic continuum Schr\"odinger operators due to Cuccagna \cite{cuccagna2008dispersion}. Dispersive estimates were subsequently proved for discrete Schr\"odinger operators by Mi and Zhao first, in period two \cite{MZ2020} and then for general periods \cite{MZ2022}; however, for periods $p \geq 3$, the scaling behaves like $t^{-1/(p+1)}$, which gets weaker as $p$ grows. In the KAM regime of small analytic quasiperiodic potentials, work of Bambusi and Zhao \cite{BambusiZhao2020Advances} finds the first dispersive estimate for aperiodic almost periodic Schr\"odinger operators.

\subsection{Results}

It is well known that solutions to \eqref{eq:timeschrod} with $V\equiv 0$ satisfy the following $\ell^1 \to \ell^\infty$ dispersive estimate, compare \cite{stefanov2005asymptotic, mielke2010dispersive}:
\begin{align} \label{eq:freeDispersion}
\|e^{-{\iop}tH_0}\psi\|_\infty\lesssim \langle t\rangle^{-1/3}\|\psi\|_1 ,
\end{align}
 where we have denoted $\langle t\rangle=\sqrt{1+t^2}$. 
Our main result is a dispersive estimate at the free rate for solutions to \eqref{eq:timeschrod} with $V$ a periodic potential.

\begin{theorem}\label{thm:periodicdispersion}
If $V:\bbZ \to \bbR$ is periodic, there is a constant $M>0$ such that
\begin{align}\label{eq:dispersiveestimate}
\| e^{-{\iop}tH_V}\psi\|_\infty \leq M \langle t\rangle^{-1/3}\|\psi\|_1
\end{align}
for all $t \in \bbR$ and all $\psi \in \ell^1(\bbZ)$.
\end{theorem}
\begin{remark}\mbox{\,}
\begin{enumerate}
\item This theorem contains and improves the results of Mi and Zhao, \cite{MZ2020,MZ2022}. 
In particular, \cite{MZ2020} shows dispersion  for $2$-periodic potentials at the established free rate \eqref{eq:freeDispersion}, and \cite{MZ2022} extends this work to show the dispersive rate of $\min\{\frac{1}{3},\frac{1}{p+1}\} $ for potentials of arbitrary period $p$. \medskip
\item 
The exponent of $1/3$ found in Theorem~\ref{thm:periodicdispersion} is the best known rate of dispersive decay estimate for the free discrete Laplacian. To the best of our knowledge, it is not known whether the exponent $1/3$ is sharp in the free case. Thus, “optimal” in the title of this work signifies that we are able to deduce a dispersive estimate at the best rate that is known for the free operator. See also the discussion in \cite{stefanov2005asymptotic}.
\medskip
\item The proof can be applied with minimal cosmetic changes to periodic \emph{Jacobi matrices}, that is, operators of the form
\begin{equation}
[J\psi](n)
= a(n-1)\psi(n-1) + b(n)\psi(n) + a(n) \psi(n+1)
\end{equation}
where $a(n)>0$, $b(n) \in \bbR$ are periodic.\medskip

\item By the Riesz--Thorin interpolation theorem, we may interpolate between
$$
\|e^{-{\iop}tH_V}\|_{\ell^1 \to \ell^\infty} \le M \langle t\rangle^{-1/3} \; \text{and} \; \|e^{-{\iop}tH_V}\|_{\ell^2 \to \ell^2} = 1
$$
to obtain
$$
\|e^{-{\iop}tH_V}\|_{\ell^p \to \ell^{p'}} \le \left( M \langle t\rangle^{-1/3} \right)^{\frac{2}{p} - 1}
$$
for $1 < p < 2$ (and $\frac{1}{p} + \frac{1}{p'} = 1$).
\end{enumerate}
\end{remark}

As in the earlier works \cite{MZ2020,MZ2022} (see also \cite{BambusiZhao2020Advances}), by standard methods, we may deduce from Theorem~\ref{thm:periodicdispersion} an $\ell^1 \to \ell^\infty$ dispersive decay estimate for the discrete nonlinear Schr\"odinger (NLS) equation:\footnote{Often in the physics literature, one works with operators of the form $-\Delta +V$, but it is easy to obtain Theorem~\ref{thm:periodicdispersion} for the other choice of sign by replacing $(t,V)$ by $(-t,-V)$.}
\begin{align}\label{eq:dnls}
{\iop}\frac{\partial\psi_n}{\partial t} = (H_V\psi)_n\pm |\psi_n|^{\sigma-1}\psi_n,\;n\in\bbZ.
\end{align}
In addition to being a natural discretization of the continuum NLS equation, the equation \eqref{eq:dnls} has found applications in the physics literature through the study of optical wave guides \cite{eisenberg2002optical,peschel2002optical}, photonic lattices \cite{efremidis2002discrete,sukhorukov2003spatial} and as a suitable model for Bose--Einstein condensates constrained by strong optical lattices \cite{cataliotti2001josephson,cataliotti2003superfluid}. We also refer the reader to the survey article \cite{kevrekidis2001discrete}. The following estimate follows quickly from Theorem~\ref{thm:periodicdispersion}, cf. \cite{BambusiZhao2020Advances}.\footnote{Mi and Zhao claim for $\sigma = p+3$, but this case seems to require a logarithmic correction.}

\begin{corollary}\label{cor:periodicNLS}
Let $\sigma> 5$ and suppose $V:\bbZ \to \bbR$ is periodic. There exist constants $C,\delta>0$ so that if $\|\psi(0)\|_1\leq \delta$, then $\psi(t)$,  the solution to the Cauchy problem for \eqref{eq:dnls} with initial datum $\psi(0)$, satisfies 
\begin{align}\label{eq:nlsdispersion}
\|\psi(t)\|_\infty\leq C\langle t\rangle^{-1/3}
\end{align}
for all $t \in \bbR$.
\end{corollary}

We note that in addition to improving the decay rate to the best established one, Corollary~\ref{cor:periodicNLS} expands the allowable nonlinearities of \cite[Corollary 1.1]{MZ2022} from $\sigma> \max\{5,p+3\}$ to $\sigma> 5$.

The rest of the paper is laid out as follows.
We recall some basic facts about periodic operators and prove some useful estimates on the Marchenko--Ostrovski mapping in Section~\ref{sec:MOmap}, which we then put to use in proving Theorem~\ref{thm:periodicdispersion} in Section~\ref{sec:PerDisp}. 

While this manuscript was being finished, we were informed that similar results for half-line operators were proved independently by different means by Kassem, Sagiv, and Weinstein \cite{KSWpreprint}. Our methods enable us to prove a global dispersive estimate for all periodic operators, whereas their methods allow them to prove global dispersive estimates for periodic operators obeying certain non-degeneracy conditions and local dispersive estimates for all periodic operators.

 \subsection*{Acknowledgements} We are grateful to the American Institute of Mathematics for hospitality during a recent SQuaRE program, at which some of this work was done.

\section{Analysis of the Marchenko--Ostrovski Mapping} \label{sec:MOmap}

To set notation, let us recall some important objects; see \cite[Section~7.2]{DF2024ESO2} for further background and proofs of relevant statements.
Given a discrete $p$-periodic Schr\"odinger operator $H = H_V$ and $z \in \bbC$, we consider the difference equation
\begin{equation} \label{eq:differenceEq}
u(n-1) + u(n+1) + V(n) u(n) = zu(n), \quad n \in \bbZ, \ u \in \bbC^\bbZ,
\end{equation}
which can be recast as a matrix recurrence, viz.:
\begin{equation}
\begin{bmatrix} u(n+1) \\ u(n) \end{bmatrix}
=
\begin{bmatrix} z-V(n) & -1 \\ 1 & 0 \end{bmatrix}
\begin{bmatrix} u(n) \\ u(n-1) \end{bmatrix}.
\end{equation}
This motivates one to define the monodromy matrix by
\begin{equation}
\Phi_V(z;p) = \begin{bmatrix} z - V(p) & -1 \\ 1 & 0 \end{bmatrix} \cdots 
\begin{bmatrix} z - V(1) & -1 \\ 1 & 0 \end{bmatrix}.
\end{equation}
Its trace $\Delta_V(z;p) := \tr \Phi_V(z;p)$ is called the \emph{Floquet discriminant} (or just the \emph{discriminant}).
When the period and potential are clear from context, we drop them from the notation and write simply $\Phi(z)$ and $\Delta(z)$.

For $k \in \bbR$, let  $H_V(k;p)$ denote the restriction of $H_V$ to the subspace $\{ \psi \in \bbC^\bbZ : \psi(n+p) \equiv e^{{\iop}pk} \psi(n)\}$.
Naturally, $H_V(k;p)$ is given by a $p \times p$ matrix; in a suitable basis:
\begin{equation} \label{eq:HkDef}
H_V(k;p) 
= 
\begin{cases}
V(1) + 2\cos(pk) & p=1, \\[2mm]
\begin{bmatrix}
V(1) & 1 + e^{-{\iop}pk} \\
1+e^{{\iop}pk} & V(2)
\end{bmatrix}  & p=2,
\\[6mm]
\begin{bmatrix}
V(1) & 1 &&& e^{-{\iop}pk} \\
1 & V(2) & 1 \\
& \ddots & \ddots & \ddots\\
&& 1 & V(p-1) & 1 \\
e^{{\iop}pk} &&& 1 & V(p)
\end{bmatrix}
& p \geq 3,
\end{cases}
\end{equation}
where unspecified entries in the case $p\geq 3$ are zero.
The eigenvalues (listed with multiplicity) of $H(k)$ will be denoted 
\begin{align*}
E_{V,1}(k;p) \leq \cdots \leq E_{V,p}(k;p),
\end{align*}
with corresponding normalized eigenvectors $\{v_{V,j}(k;p) : j=1,2,\ldots, p\}$.
As before, we suppress $V$ and $p$ when they are clear from context and simply write $H(k)$, $E_j(k)$, and $v_j(k)$.
Since $H(-k)$ is anti-unitarily equivalent to $H(k)$ via complex conjugation, we note that $E_j(-k) = E_j(k)$ for each $j$ and $k$.

Due to eigenvalue perturbation theory, we can (and do) choose the eigenvectors $\{v_j(k)\}$ in such a way that $v_j$ is continuous on $[0,\pi/p]$ and analytic on $(0,\pi/p)$ \cite{Kato}.
Furthermore, each $E_j$ also has an analytic extension through the endpoints $k_*=0,\pi/p$.
This is clear when $E_j$ is simple at the endpoint $k_*$, but one must be a bit careful in degenerate situations:
specifically, if $E_j(k_*)$ is doubly degenerate at $k_* \in \{0,\pi/p\}$, the analytic extension of $E_j$ might be given by one of the other band functions; compare the discussion on \cite[pp.\ 258--259]{Simon2011Szego}.
We write $[v_j(k)]_q = \langle v_j(k), e_q \rangle$ for the $q$th entry of $v_j(k)$; 
here and throughout the paper, $\langle \cdot,\cdot \rangle$ denotes the standard inner product that is antilinear in the second coordinate and $\{e_q: q=1,\ldots,p\}$ denotes the standard basis of $\bbC^p$.

The spectrum $\sigma(H)=\Sigma$ is then given by 
\begin{align*}
    \Sigma=\bigcup_{j=1}^p [\lambda_{2j-1},\lambda_{2j}],
\end{align*}
where $\lambda_{2j-1} = \min E_j(k)$ and $\lambda_{2j} = \max E_j(k)$.
The function $E_j$ is monotonic on $[0,\pi/p]$ and moreover
\begin{equation} \label{eq:monotonicityOfEj}
E_j \text{ is }
\begin{cases}
 \text{decreasing on } [0,\pi/p] & \text{if } p-j \text{ is even,} \\
 \text{increasing on } [0,\pi/p]  & \text{if } p-j \text{ is odd.}
 \end{cases}
\end{equation} 
Thus, we have $[\lambda_{2j-1},\lambda_{2j}] = [E_j(0), E_j(\pi/p)]$ if $p-j$ is odd and  $[E_j(\pi/p), E_j(0)]$ otherwise.

The quantities above are related by
\begin{equation}
\det(z-H(k))
= \Delta(z) - 2\cos(pk).
\end{equation}
In particular, $\Delta(E_j(k)) = 2\cos(pk)$ for any $j=1,2,\ldots, p$ and $k \in \bbR$.

Since $z \in \bbC$ belongs to $\Sigma$ if and only if $\Phi(z)$ has spectral radius one, there is an analytic function $\Theta:\bbC_+ \to \bbC_+:= \{z : \imaginary z >0\}$, called the Marchenko--Ostrovski mapping (see \cite[Section~10.10]{LukicBook}), such that $\Re(\Theta)\in [-\pi,0]$ and
\begin{equation}\label{eq:discriminantMO}
\Delta(z)  =2\cos(p\Theta(z)), \quad z \in \bbC_+.
\end{equation}

The Herglotz function $\Theta$ then satisfies
\begin{align} \label{eq:ThetaEinverse}
    \Theta(E(k))=k,\;k\in [-\pi,0],
\end{align}
where $E(k)$ is made up of the traditional band functions $E_j(k)$, $1\leq j\leq p$, $k\in [0,\pi/p]$ as follows:
since for $p-j$ odd, $E_j$ increases, while for $p-j$ even, $E_j$ decreases, we have 
\begin{align*}
    E(k)=\begin{cases}
        E_j\left(k+\pi-\frac{(j-1)\pi}{p}\right),& -\pi+\frac{(j-1)\pi}{p}\leq k< -\pi+\frac{j\pi}{p},\;p-j\;\text{odd}\\
        E_j\left(\frac{j\pi}{p}-k-\pi\right),& -\pi+\frac{(j-1)\pi}{p}\leq k< -\pi+\frac{j\pi}{p},\;p-j\;\text{even}.
    \end{cases}   
\end{align*}
 Differentiating the relation \eqref{eq:discriminantMO} yields
\begin{align}\label{eq:dkappa}
\begin{split}
\Theta'(z)
&=\frac{{\iop}\Delta'(z)}{p\sqrt{\Delta(z)^2-4}},
\quad z\in \bbC_+.
\end{split}
\end{align}
 One also has the following product formulas:
\begin{align}\label{eq:products}
\begin{split}
    &\Delta(z)^2-4=\prod_{j=1}^{p}(z-\lambda_{2j-1})(z-\lambda_{2j})\\
    &\Delta'(z)=p\prod_{j=1}^{p-1}(z-\kappa_j),
\end{split}
\end{align}
where $\kappa_1 < \cdots < \kappa_{p-1}$ denote the zeros of $\Delta'$, which satisfy $\lambda_{2j} \leq \kappa_j \leq \lambda_{2j+1}$ for each $j$. This yields
\begin{align}\label{eq:dkproduct}
\Theta'(z)
&=\sqrt{\frac{1}{(z-\lambda_1)(\lambda_{2p}-z)}\prod_{j=1}^{p-1}\frac{(z-\kappa_j)^2}{(z-\lambda_{2j})(z-\lambda_{2j+1})}},
\quad z\in\bbC_+
\end{align}
with argument chosen so that $\arg(\Theta')=\frac{\pi}{2}$ on the gap $z\in (\lambda_{2p},\infty)$. 
Finally, $\Theta$ extends continuously to $\overline{\bbC_+}$, with $\Im(\Theta(z))=0$ for $z\in \Sigma$, so that by the Schwarz reflection principle, $\Theta$ admits an analytic extension through any connected component of the interior of $\Sigma$.
We note that one may also extend $\Theta$ through the gaps, and these extensions through each gap differ by an additive real constant.\footnote{We also note that by construction, $e^{\pm {\iop} p \Theta(z)}$ are the eigenvalues of $\Phi(z)$, so the spectral radius of $\Phi(z)$ is given by $e^{p \Im \Theta(z)}$ for $z \in \bbC_+$, which in particular implies that $\Im \Theta(z) = L(z)$, the Lyapunov exponent.}

Using these facts, we may prove the following lemma. For a related continuum result for the ``global" quasimomentum, see \cite{Korotyaev1992}, using some properties of the global energy found in \cite{firsova1987direct}. While similar, both the proof and conclusion in that work differ from what is proved below both in setting and in our definition of the quasimomentum through the Marchenko--Ostrovski mapping, $\Theta$. 

\begin{lemma}\label{lem:BorgMarchenkointerior}
For $x\in \mathring{\Sigma}$,
\begin{align*}
\Theta'(x),\Theta^{(3)}(x)>0
\end{align*}
and for each connected component $[\lambda_{m},\lambda_{n}]$ of $\Sigma$, we have
\begin{align*}
&\lim_{x\to \lambda_{m}^+}\Theta''(x)
=-\infty, \lim_{x\to\lambda_{n}^-}\Theta''(x)=+\infty.
\end{align*}
Thus, $\Theta''(x)$ has a unique zero $x_*$ in each connected component of $\mathring{\Sigma}$ and $\Theta^{(3)}(x_*)\ne 0$.
\end{lemma}

\begin{proof}
Since $\Theta$ is Herglotz, we may write:

\begin{align*}
    \Theta(z)=az+b+\int\limits_{\bbR}\frac{1+\lambda z}{\lambda-z} \, \rmd \nu(\lambda)
\end{align*}
for $a>0$, $b\in \bbR$ and $\rmd \nu$ a finite positive measure on $\bbR$. Furthermore, $a$ may be computed as 
\begin{align*}
a=\lim_{y\to\infty}\frac{\Theta(\iop y)}{\iop y}=0,
\end{align*}
since $\Theta(z)=\iop\log(z)+o(1)$ as $z\to \infty$ nontangentially in $\mathbb C_+$ (cf.~\cite[Lemma~10.65]{LukicBook}). 

Taking imaginary parts of the Herglotz representation, we have 
\begin{align*}
\Im(\Theta(x+{\iop}y))
&=\int\limits_{\bbR}\frac{y(1+\lambda^2)}{(\lambda-x)^2+y^2} \, \rmd \nu(\lambda)\\
&=\int\limits_{\bbR}\frac{y}{(\lambda-x)^2+y^2} \, \rmd \mu(\lambda)
\end{align*}
for the Poisson finite measure $\rmd \mu(\lambda):=(1+\lambda^2) \, \rmd \nu(\lambda)$. Differentiating this expression in the upper half plane, we have by the Cauchy-Riemann equations
\begin{align}\label{eq:CReqs}
\frac{\partial \Re(\Theta)}{\partial x}
&=\int\limits_{\bbR}\frac{\rmd \mu(\lambda)}{(\lambda-x)^2+y^2}-2y^2\int\limits_{\bbR}\frac{\rmd \mu(\lambda)}{((\lambda-x)^2+y^2)^{2}}. 
\end{align}

Since $\Im(\Theta)$ extends continuously to $\overline{\bbC_+}$ we may compute $\rmd \mu$ by Stieltjes inversion:
\begin{align*}
    \rmd \mu(\lambda)
    =\frac1\pi \Im(\Theta(\lambda))\chi_{\bbR\setminus \Sigma}(\lambda) \, \rmd \lambda,
\end{align*}
where we have used that $\Im(\Theta(\lambda))=0$ for $\lambda\in \Sigma$. 
Since we may extend $\Theta$ analytically through the interior of a connected component of $\Sigma$,  for $x \in \mathring{\Sigma}$, we have 
\begin{align*}
\Theta'(x)=\frac1\pi \int\limits_{\bbR\setminus \Sigma}\frac{\Im(\Theta(\lambda))}{(\lambda-x)^2} \, \rmd \lambda
\end{align*}
by taking $y\downarrow 0$ in \eqref{eq:CReqs}. Differentiating twice more, we have 
\begin{align*}
    \Theta^{(3)}(x)=\frac6\pi \int\limits_{\bbR\setminus \Sigma}\frac{\Im(\Theta(\lambda))}{(\lambda-x)^4} \, \rmd \lambda.
\end{align*}
Thus, $\Theta'(x),\Theta^{(3)}(x)>0$ since
\begin{align*}
    \Im(\Theta(\lambda))>0,\;\lambda\in \bbR\setminus \Sigma,
\end{align*} 
cf.\ \cite{LukicBook}. It remains to verify the behavior at the edges. Examining boundary values of \eqref{eq:dkproduct}, for $x$ near $\lambda_k$, we may write:
\begin{align*}
\Theta'(x)=g(x)|x-\lambda_k|^{-1/2}
\end{align*}
for $g$ smooth in a neighborhood $\lambda_k$. Thus,
\begin{align*}
\Theta''(x)= g'(x)|x-\lambda_k|^{-1/2}\pm \frac12g(x)|x-\lambda_k|^{-3/2},
\end{align*}
for a sign dependent on whether $\lambda_k$ is a left or right endpoint.
So, we have $\lim_{x\to\lambda_k^\pm}|\Theta''(x)|=\infty$. Since $\Theta^{(3)}>0$, the desired limits follow.
\end{proof}

Using the previous lemma, we are able to show that for each $j$, $E_j''$ and $E_j'''$ cannot vanish simultaneously in $[0,\pi/p]$, which is the crucial input needed to apply the van der Corput lemma and estimate the oscillatory integrals later in the manuscript.
In particular, this allows us to improve the ``$p+1$''  to ``$3$'' in the main dispersive estimate.
We will exploit the inverse function $E$ from \eqref{eq:ThetaEinverse}: indeed, we note that $E''(k')=E^{(3)}(k')=0$ if and only if $E_j''(k)=E_j^{(3)}(k)=0$ for some $j$, where the derivatives at $k'\in \{ -\pi+\frac{\ell \pi}{p}\}_{\ell=0}^p$ or $k\in\{0,\pi/p\}$ are understood in the limit.

\begin{corollary}\label{cor:evalslower}
For any $V: \bbZ \to \bbR$ that is $p$-periodic, 
\begin{align}\label{eq:bandfcnmins}
  \delta(V)
  :=  \min_{1\leq j \leq p}\min_{k\in [0,\pi/p]}(|E_{V,j}''(k;p)|+|E_{V,j}^{(3)}(k;p)|) >0.
\end{align}
\end{corollary}

\begin{proof}
We note that for $k$ such that $E(k)\in \mathring{\Sigma}$, differentiating the relation $\Theta(E(k)) = k$ and rearranging yields
\begin{align*}
E''(k)&=\frac{-\Theta''(E(k))}{(\Theta'(E(k)))^3}\\
E^{(3)}(k)&=\frac{1}{(\Theta'(E(k)))^5}(3[\Theta''(E(k))]^2-\Theta^{(3)}(E(k))\Theta'(E(k)))
\end{align*}
so that $E''(k)=E^{(3)}(k)=0\iff \Theta''(E(k))=\Theta^{(3)}(E(k))=0$. By Lemma~\ref{lem:BorgMarchenkointerior}, this cannot occur away from an open gap, so it remains to check at band edges near an open gap. However, at an open gap, $\Delta'\ne 0$ and we may differentiate the relation $\Delta(E_j(k)) = 2\cos (pk)$, rearrange, and evaluate for $k_*\in\{0,\pi/p\}$ to obtain
\begin{align}\nonumber
E_j''(k_*)&=\frac{2p\sin(pk_*)\Delta''(E_j(k_*))E_j'(k_*)-2p^2\cos(pk_*)\Delta'(E_j(k_*))}{\Delta'(E_j(k_*))^2}\\
\label{eq:ejdoubleprime}
&=\frac{\pm 2p^2}{\Delta'(E_j(k_*))}\ne 0.
\end{align}
Consequently, \eqref{eq:bandfcnmins} follows by continuity of each band function $E_j$ and its derivatives on the interval $[0,\pi/p]$. 
\end{proof}

\section{Optimal Dispersion for Periodic Operators} \label{sec:PerDisp}
Let us begin by recalling the decomposition of $H$ as a direct integral by means of the $p$-Fourier transform.
Put $\bbB = [-\pi/p,\pi/p)$ and define an isometry 
\begin{align*}
    \cF_p:\ell^2(\bbZ)\to L^2(\bbB;\bbC^p),
\end{align*}
initially for $\psi \in \ell^1(\bbZ)$ by
\begin{align*}
    \sum_{r\in \bbZ}\psi_{q+rp}e^{-{\iop}r p k}
    =\langle (\mathcal F_p\psi)(k),e_q\rangle =: [\hat{\psi}(k)]_q, \quad k\in \bbB,1\leq q\leq p,
\end{align*}
 and then extended to $\ell^2(\bbZ)$.
Then, we have the direct integral representation
\begin{align*}
    H=\cF_p^{-1}\int\limits_{\bbB}^\oplus H(k)\frac{\rmd k}{| \bbB|}\cF_p, 
\end{align*}
where
$
    H(k):\bbC^p\to \bbC^p
$
is as in \eqref{eq:HkDef}. 
We may write
\begin{align*}
    H(k)=\sum_{j=1}^pE_j(k)P_j(k)
\end{align*}
for orthogonal projections $P_j$ having one-dimensional range except possibly when $k\in\{-\pi/p,0\}$. 

We have the following elementary pointwise representation for the propagation of $\psi\in \ell^1$.
We give the proof to keep the paper self-contained.

\begin{lemma}\label{lem:propagator}
Let $\psi\in \ell^1(\bbZ)$, and let $n=m+\ell p$ for some $1\leq m\leq p$ and $\ell\in\bbZ$. We have
\begin{align*}
    \langle e^{-{\iop}tH}\psi,\delta_n\rangle
    =\sum_{q=1}^p \sum_{r\in \bbZ}\psi_{q+rp}\int\limits_{ \bbB}\sum_{j=1}^pe^{-{\iop}(tE_j(k)+p(r-\ell) k)}[v_j(k)]_m \overline{[v_j(k)]_q} \, \frac{\rmd k}{| \bbB|}.
\end{align*}
\end{lemma}

\begin{proof}
We may write
\begin{align*}
    e^{-{\iop}tH}\psi 
    & =\mathcal F_p^{-1}\int\limits_{ \bbB}^\oplus e^{-{\iop}tH(k)}\hat{\psi}(k) \, \frac{\rmd k}{| \bbB|} \\
   & =\mathcal F_p^{-1}\int\limits_{ \bbB}^\oplus\sum_{j=1}^pe^{-{\iop}tE_j(k)}\langle \hat{\psi}(k),v_j(k)\rangle v_j(k)  \, \frac{\rmd k}{| \bbB|}
\end{align*}
since away from $k\in \{0,\pi/p\}$, the projections have one-dimensional image and hence can be written
\begin{align} \label{eq:PjkDef}
P_j(k)\hat\psi=\langle\hat \psi(k), v_j(k)\rangle v_j(k)    .
\end{align}

So, we may use unitarity of the $\mathcal F_p$ transform to compute 
\begin{align*}
    \langle e^{-{\iop}tH}\psi,\delta_n\rangle
    &=\left \langle\; \int\limits_{ \bbB}^\oplus e^{-{\iop}tH(k)}\hat{\psi}(k)\frac{\rmd k}{| \bbB|}, \hat{\delta}_n(k)\right \rangle\\
    &= \int\limits_{ \bbB}\sum_{j=1}^pe^{-{\iop}tE_j(k)}\langle v_j(k),\hat{\delta}_n(k)\rangle \langle \hat{\psi}(k),v_j(k)\rangle \frac{\rmd k}{| \bbB|}.
\end{align*}
Writing $n=m+\ell p$ for an $\ell\in \bbZ$ and $1\leq m\leq p$, we have $\hat{\delta}_n = e^{-{\iop}\ell pk}e_m$ by definition. 
Since $\langle v_j(k),e_m\rangle =[v_j(k)]_m$, we have
\begin{align*}
    \langle e^{-{\iop}tH}\psi,\delta_n\rangle
    &= \int\limits_{ \bbB}\sum_{j=1}^pe^{-{\iop}(tE_j(k)-p\ell k)}[v_j(k)]_m \langle \hat{\psi}(k),v_j(k)\rangle \, \frac{\rmd k}{| \bbB|}\\
    &= \int\limits_{ \bbB}\sum_{j=1}^pe^{-{\iop}(tE_j(k)-p\ell k)}[v_j(k)]_m\sum_{q=1}^p[\hat{\psi}(k)]_q \overline{[v_j(k)]_q} \, \frac{\rmd k}{| \bbB|}\\
    &=\sum_{q=1}^p \sum_{r\in \bbZ}\psi_{q+rp} \int\limits_{ \bbB}\sum_{j=1}^pe^{-{\iop}(tE_j(k)+p(r-\ell) k)}[v_j(k)]_m \overline{[v_j(k)]_q} \, \frac{\rmd k}{| \bbB|}
\end{align*}
by the dominated convergence theorem. 
\end{proof}

\begin{proof}[Proof of Theorem~\ref{thm:periodicdispersion}]
Write $n = m + \ell p$ as in Lemma~\ref{lem:propagator}.
Using the expression given by the lemma, we have 
\begin{align*}
 &   |\langle e^{-{\iop}tH}\psi,\delta_n\rangle| \\
 & \qquad \leq 
    \sum_{q=1}^p \sum_{r\in \bbZ}|\psi_{q+rp}| \sum_{j=1}^p\left|\;\int\limits_{ \bbB}e^{-{\iop}(tE_j(k)+p(r-\ell) k)}[v_j(k)]_m \overline{[v_j(k)]_q} \, \frac{\rmd k}{| \bbB|}\right |.
\end{align*}
Using Corollary~\ref{cor:evalslower} and $E_j(k)=E_j(-k)$, we may decompose for each $1\leq j\leq p$, 
\begin{align*}
     \bbB
     =\{k\in  \bbB: |E_j^{(3)}(k)|\geq \delta/2\}\cup \{k\in  \bbB: |E_j^{(2)}(k)|\geq \delta/2\},
\end{align*}
so we define
\begin{equation}
K_j^2 = \{k\in  \bbB: |E_j^{(2)}(k)|\geq \delta/2\}, \quad
K_j^3 = \bbB \setminus K_j^2,
\end{equation}
where we assume $\delta<1$ for convenience in what follows. 
Let $t$ be given with $|t|\geq 1$. Due to \cite[Lemma~3]{Eliasson1997Acta}, each $K_j^i$ has finitely many components.
We denote this number by $\# K_j^i$, and define 
\begin{align*}
    C_{V,j}
    :=\sum_{s=1}^{\# K_j^2+\# K_j^3 } |[v_j(k_s)]_m \overline{[v_j(k_s)]_q}| +   \int\limits_{ \bbB}\left|\frac{\rmd}{\rmd k}\left( [v_j(k)]_m \overline{[v_j(k)]_q} \right) \right| \, \rmd k
\end{align*}
where $\{k_s\}$ denotes the collection of all left endpoints of the intervals constituting $K_j^2$ and $K_j^3$.

Then, we split the integrals above and estimate using the van der Corput lemma (compare Lemma~\ref{lem:VanderCorput}) to find 
\begin{align*}
&\left|\;\int\limits_{ \bbB}e^{-{\iop}(tE_j(k)+p(r-\ell) k)}[v_j(k)]_m \overline{[v_j(k)]_q} \, \frac{\rmd k}{| \bbB|}\right | \\
& \qquad \leq \sum_{i\in \{2,3\}} \frac{p}{2\pi}\left|\;\int\limits_{K_{j}^i}e^{-{\iop}(tE_j(k)+p(r-\ell) k)}[v_j(k)]_m \overline{[v_j(k)]_q} \, \rmd k\right |\\
& \qquad \leq \frac{p C_3}{2\pi }C_{V,j}((\delta/2)^{-1/2}|t|^{-1/3}) \\[2mm]
& \qquad \leq \frac{p C_3}{2^{1/3}\pi }C_{V,j}(\delta^{-1/2}\langle t\rangle ^{-1/3}),
\end{align*}
where we note that $C_3 = 18>8=C_2$ from Lemma~\ref{lem:VanderCorput} and have used $\delta<1$ and $|t|\geq 1$, so that $\sqrt{2}|t|\geq \langle t\rangle$.

We may further estimate $\sum_{j=1}^pC_{V,j}$ by noting that for all $k\in \bbB$, the $v_j(k)$ form an orthonormal basis, and for each 
\begin{align*}
1\leq s\leq \# K_j^2+\# K_j^3,
\end{align*}
we have 
\begin{align*}
\sum_{j=1}^p|[v_j(k_s)]_m [v_j(k_s)]_q|\leq \frac12\sum_{j=1}^p (|\langle v_j(k_s),e_m\rangle|^2+|\langle v_j(k_s),e_q\rangle|^2)=1.
\end{align*}
Similarly,
\begin{align*}
\int\limits_{ \bbB}\left|\frac{\rmd}{\rmd k}\left( [v_j(k)]_m \overline{[v_j(k)]_q} \right) \right|\, \rmd k
&\leq \frac12 \int\limits_{ \bbB}(|\langle v_j'(k),e_m\rangle|^2+|\langle v_j'(k),e_q\rangle|^2 \, \rmd k \\
&\qquad + \frac12\int\limits_{\bbB} |\langle v_j(k),e_q\rangle|^2+|\langle v_j(k),e_m\rangle|^2) \, \rmd k\\
&\leq \int\limits_{ \bbB}\| v_j'(k)\|^2\, \rmd k+\frac12\int\limits_{\bbB} |\langle v_j(k),e_q\rangle|^2+|\langle v_j(k),e_m\rangle|^2\, \rmd k
\end{align*}
by the Cauchy--Schwarz inequality. Thus, 
\begin{align*}
&\sum_{j=1}^p\int\limits_{ \bbB}\left|\frac{\rmd}{\rmd k}\left( [v_j(k)]_m \overline{[v_j(k)]_q} \right) \right|\, \rmd k \\
& \qquad\leq \sum_{j=1}^p\int\limits_{ \bbB}\| v_j'(k)\|^2\, \rmd k+\sum_{j=1}^p\frac12\int\limits_{\bbB} |\langle v_j(k),e_q\rangle|^2+|\langle v_j(k),e_m\rangle|^2\, \rmd k\\
&\qquad=\sum_{j=1}^p\int\limits_{ \bbB}\| v_j'(k)\|^2\, \rmd k+\frac{2\pi}{p}\\
&\qquad \leq p\max_{1\leq j\leq p}\int\limits_{ \bbB}\| v_j'(k)\|^2\, \rmd k+2\pi/p,
\end{align*}
using again that the $v_j$ form an orthonormal basis for every $k \neq -\pi/p, 0$.

So, putting these estimates together, we have  
\begin{align*}
\sum_{j=1}^p C_{V,j}\leq \max_{1\leq j\leq p} (\#K_j^{3}+\#K_j^2)\left(1+ p\max_{1\leq j\leq p}\int\limits_{ \bbB}\| v_j'(k)\|^2\, \rmd k+2\pi/p \right)=:C_V.
\end{align*}

Thus, noting that 
\begin{align*}
    \sum_{q=1}^p \sum_{r\in \bbZ}|\psi_{q+rp}|=\|\psi\|_1
\end{align*}
yields \eqref{eq:dispersiveestimate} for $|t|\geq 1$ with constant
\begin{align*}
    \frac{p C_3}{2^{1/3}\pi }C_{V}\delta^{-1/2}.
\end{align*}

To deal with $|t| < 1$, note that
\begin{align*}
 \| e^{-{\iop}tH}\psi\|_\infty
 \leq \| e^{-{\iop}tH}\psi\|_2
= \|\psi\|_2
 \leq \|\psi\|_1
 \leq 2^{1/6} \langle t \rangle^{-1/3} \|\psi\|_1,
 \end{align*}
 since $\langle t \rangle < \langle 1 \rangle =2^{1/2}$ for $|t| < 1$.

Thus, with
\begin{align} \label{eq:dispersiveConstantMVP}
    M=M_V=\max\left\{     \frac{p C_3}{2^{1/3}\pi }C_{V}\delta^{-1/2}, 2^{1/6} \right\},
\end{align}
the desired estimate is proved.
\end{proof}

\begin{appendix}
\section{Van der Corput Lemma}
We need the following modified van der Corput lemma, proved in much the same way as the standard form, cf.  \cite{stein1993harmonic}.

\begin{lemma}\label{lem:VanderCorput}
      Suppose $\phi\in C^k(a,b)$ is real-valued with $|\phi^{(k)}(x)|\geq \delta>0$ for $k\geq 2$ on $(a,b)$, while $\eta$ is real valued with $\eta''(x)\equiv 0$ on $(a,b)$. Then, for $\psi\in C^1(a,b)\cap C([a,b])$, and $\lambda\ne 0$
    \begin{align}\label{eq:vandercorput}
        \left|\int\limits_{a}^b e^{{\iop}(\lambda\phi(x)+\eta(x))}\psi(x) \, \rmd x \right|\leq 
        C_k\delta^{-1/k}\left(|\psi(a)|+\int\limits_a^b |\psi'(x)|\, \rmd x\right)|\lambda|^{-1/k}
    \end{align}
    where
    \begin{align*}
        C_k=5\cdot 2^{k-1}-2.
    \end{align*}
\end{lemma}

\begin{proof}
Defining the phase $p_\lambda(x)=\phi(x)+\frac1\lambda \eta(x)$, and rewriting the left hand side of \eqref{eq:vandercorput}, 
\begin{align*}
  \left|\int\limits_{a}^b e^{{\iop}\lambda p_\lambda(x)} \psi(x)\, \rmd x \right|
\end{align*}
the proof follows as in the standard case since $|p_\lambda^{(k)}(x)|=|\phi^{(k)}(x)|\geq \delta$ for $k\geq 2$.
\end{proof}

\end{appendix}

\bibliographystyle{abbrv}
\bibliography{lit}

\end{document}